\RequirePackage{fix-cm}
\documentclass[smallextended,draft]{svjour3}       

\smartqed  
\usepackage{amssymb}
\usepackage[active]{srcltx}
\usepackage[english]{babel}
\usepackage{amsmath}
\newtheorem{de}{Definition}
\newtheorem{Corollary}{Corollary}

\newcommand{\dgam}{d_{\gamma}}

 \journalname{}
\begin{document}
\begin{center}{{\LARGE {\bf Lion and Man Game in Compact Spaces}}\\
\ \\
Yufereva Olga\\ 
yufereva12@gmail.com}
\end{center}


\author{Yufereva Olga}

\institute{O. Yufereva \at
              Krasovskii Institute of Mathematics and Mechanics UrB RAS, 16, Kovalevskoy str., Ekaterinburg 620990, Russia \\
              \email{yufereva12@gmail.com}             \\
              Chair of Applied Mathematics and Mechanics, Institute of Mathematics and Computer Science, Ural Federal University, 4, Turgeneva str.,  Ekaterinburg 620083, Russia
         }

\begin{abstract}
The pursuit-evasion game with two persons is considered. Both players are moving in a metric space, have equal maximum speeds and  complete information about the location of each other.  We study the sufficient conditions for a capture (with a positive capture radius). We assume that Lion wins if he manages the capture independently of the initial positions of the players and the evader's strategy. We prove that  the discrete-time simple pursuit strategy is a Lion's winning strategy in a compact geodesic space $(K, d)$ satisfying the betweenness property. 
 In particular, it means that Lion wins in compact CAT(0)-spaces,  Ptolemy spaces, Buseman convex spaces or any geodesic space with convex metric.
 We also do not need to use such properties as  finite dimension, smoothness, boundary regularity  or contractibility of the loops. 

\end{abstract}
\section{Introduction}

In this article we study a pursuit-evasion game known as lion-and-man game.
Both players have equal capabilities and opposite goals. Namely Lion wants to catch Man, and Man wants to evade. There are many papers on pursuit-evasion games; see \cite{isaacs1999differential,petrosjan1993differential,pontryagin1966theory,chernous1976problem,ivanov1981optimality} for example. There is also a modern survey in \cite{kumkov}. 

 In the classical case the phase space is the unit closed disk in~$\mathbb{R}^{2}$. Besicovitch showed that the evader can escape the pursuer infinitely long \cite{littlewood1986mathematician}. But the pursuer can achieve arbitrarily small  positive distance to the evader,
for instance, using simple pursuit strategy. It is called $\varepsilon$-capture. We further say that Lion wins if he is able to 
get $\varepsilon$-capture.
Such an approach commonly appears in applications.

With regard to applications, lion-and-man game is used, e.g., by \cite{bramson2014rubber} to study Brownian motions and to show  relations between this game and such  metric properties as contractibility of loops.
 Let us further note that works like \cite{noori2015lion} may benefit the robotics community.
 Similar papers \cite{karnad2009lion,tovar2008visibility,stiffler2012shortest}, which are devoted to this area,  differ in the visibility of players or capabilities of their motions.

Let us point out several directions of theoretical researches. One specific pursuer's strategy on convex terrains  is described in  \cite{noori2015lion}.
 Another Lion's strategy was introduced in \cite{sgall2001solution} to consider a pursuit within the non-negative quadrant of the plane. This was a discrete-time version that later was generalized for CAT(0)-spaces in \cite{beveridge2015two} (CAT(0)-space is an Alexandrov space with non-positive curvature, see more in \cite{bridson2011metric}). There is an important result of \cite{alexander2010total}: the simple pursuit leads to capture in compact CAT(0)-spaces; this paper became a classical one. In \cite{bramson2014rubber}, this result was expanded on finite-dimension  CAT($\kappa$)-spaces of sufficiently small diameters.

In this article,  
   we show that simple pursuit strategy implies $\varepsilon$-capture even in geodesic compact spaces with such a geometric property as convex metric or, more general, in compact spaces satisfying the betweenness property (see Definition \ref{betweenness}). As a corollary, we have $\varepsilon$-capture in compact CAT(0)-spaces,  Ptolemy spaces, Buseman convex spaces. 
It means  
   our result implies the result from \cite{alexander2010total} for compact CAT(0)-spaces, the result of \cite[Theorem 4.6]{bramson2014rubber} for finite-dimension  CAT($\kappa$)-spaces of sufficiently small diameter. We also provide an example of the space that is not CAT(0), but still satisfies the betweenness property (Section \ref{sec:2}, example 3).

More general case --- topological spaces are studied in the recent work \cite{Barmak.2017}. But in such a way players may move along any continuous curves (speeds can not be calculated) and a capture means the equality of players' positions. 
 It brings to the existing of winning non-anticipative strategies of both players in the game on a circle (by Proposition 1 and 2 from \cite{Barmak.2017}). Similar situation (but in metric spaces) was deeply considered in the work \cite{Bollobas.2012}, which, in particular, offers ways to avoid such problems. One of them consists in the using  a discrete-time  strategy (at least by a one player). We follow this way employing a discrete-time simple pursuit strategy.


This  paper is organized as follows.
In Section 2 we present simple examples of the game. Section~3 is devoted to properties of geodesic segments (including betweenness property). In Subsection~\ref{sect of dyn}~and~\ref{sect of strategy} we present  constraints on a metric space, dynamics and Lion's strategy. In Subsection~\ref{sect of results} we formulate the main results: Proposition~\ref{th2} as describing of trajectories' behaviour, Theorem~\ref{th betw} as a  result of game theory and Corollary~\ref{th convex} as an important special case. The~whole Section~5 is devoted to the proof of Proposition~\ref{th2}.

\section{Simple Examples}
\label{sec:2}
We suppose that Lion wins if at some time moment $t^{\ast}\in [0, +\infty)$ the distance between $x(t^{\ast})$ and $ y(t^{\ast})$ is less than $ \varepsilon.$  Man wins if such $t^{\ast}$ does not exist.
Consider the games corresponding to the following dynamics of players:
 \begin{eqnarray} 
 \label{dynamic}
\begin{array}{lcl}
Lion: \quad &\dot{x} = u, \qquad &x(0) = x^{0}, \\
Man: \quad &\dot{y} = v, \qquad &y(0) = y^{0},  \\
&x, y \in K \subset \mathbb{R}^n, & \\
&u, v \in Q \subset \mathbb{R}^n, &
\end{array}
 \end{eqnarray}
 the sets $K$ and $Q$ are defined specially in the each following example.

\paragraph{Example 1} Let $K$ be the plane $\mathbb{R}^{2}, \quad
Q = \{(q_1, q_2) \in \mathbb{R}^2  \ | \ q_1^2 + q_1^2 \leq 1  \}.$
\\
It is easy to see that Man wins by moving in the direction away from the Lion's initial position.
A similar strategy leads to the Man's win in a lot of unbounded sets $K$.  Though an example of capture on an unbounded space is reported \cite{bavcak2012note},  we will later assume the compactness of~$K,$ as in~\cite{alexander2010total,bramson2014rubber}.

\paragraph{Example 2} Put
$$K = \{(x_1, x_2)\in \mathbb{R}^2 \mid x_1^2+x_2^2\leq 1 \}, \quad 
 Q = \{(q_1, q_2) \in \mathbb{R}^2  \ | \ q_1^2 + q_1^2 \leq 1  \}.$$
Here Lion wins both if he uses continuous-time simple pursuit strategy (\cite{alexander2010total}) and if he uses discrete-time one (by Theorem \ref{th betw}).

\paragraph{Example 3} Put
$$K = \{(x_1, x_2)\in \mathbb{R}^2 \mid x_1^2+x_2^2\leq 1 \}, \quad 
 Q = \{(q_1, q_2) \in \mathbb{R}^2  \ | \ \max\{|q_1|, |q_2|\} \leq 1  \}.$$

This kind of the set $Q$ generates $l_{\infty}$-metric. So, we may suppose that players move in the metric space $(K, d),$ where $d$ is $l_{\infty}$ metric.
Note that $d$ is a convex metric. Then Lion wins here by Corollary \ref{th convex}.

Unlike Example 2,  \cite{alexander2010total} can not be employed to this example, cause $(K, d)$ is not an  CAT(0)-space by the following reason.  Every CAT(0)-space $(K', d')$ guarantees the Ptolemy inequality: 
\begin{eqnarray*}
d'(x, z)d'(y, w) \leq d'(x,y)d'(z,w) + d'(x, w)d'(y, z)
\end{eqnarray*}
for every points $x, y, z, w \in K',$ that here does not hold for $x = (0,1), y=(1, 0), z=(0, -1), w=(0, -1)$ for instance.

\section{Geodesic Segments}
\label{Section of Geodesic Segment}

In this section we recall some definiton and properties of geodesic spaces.  The definitions from the list  below  can be found in \cite{bridson2011metric} or \cite{papadopoulos} too. 
\begin{itemize}
        \item  A map $\gamma \colon [a,b]\subset \mathbb{R} \to K$ with $\gamma(a) = A, \gamma(b)=B$ is called  a {\it geodesic path joining $A$ to $ B $} (for $A,B\in K$) if $d \bigl(\gamma(t), \gamma(t')\bigr) = |t-t'|$ for all $t, t' \in [a,b]$.
        \item A metric space $(K, d)$ is called {\it a geodesic space} if every pair of points $A$, $B \in K$ can be joined by a geodesic path.
        \item The image of a geodesic path $\gamma$ joining $A$ to $ B $ is called {\it a geodesic segment  with the endpoints $A$ and $B$}. 
		\item  For any geodesic segment with the endpoints $A$ and $B$ and a point $C$ from it we have the equality $d(A, B) = d(A, C) + d(C, B).$
        \item We say $C$ {\it lies between $A$ and $B$} if the equality $d(A, B) = d(A, C) + d(C, B)$ holds.
\end{itemize}
In geodesic space we also have that if a point $C$ lies between points $A$ and $B,$ then there exists a geodesic segment with the endpoints $A$ and $B$ that contains the point $C.$
\begin{definition}
\label{betweenness}
A geodesic space $(K, d)$ is called  satisfying the betweenness property if for every four pairwise distinct points $A, B, C, D \in K$
the following implication holds true: if 
\begin{eqnarray*}
B \mbox{ lies between } A \mbox{ and }  C \quad \bigl(\mbox{i.e. } d(A, B)+ d(B, C) = d(A, C)\bigr), \\
C \mbox{ lies between } B \mbox{ and }  D  \quad \bigl(\mbox{i.e. } d(B, C)+ d(C, D) = d(B, D)\bigr),
\end{eqnarray*}
then $B$ and $C$ lie between $A$ and $D.$
\end{definition}

\begin{remark}
\label{geod-trans}
Let $(K,d)$ is a  geodesic space satisfying the betweenness property. If, for all $A, B, C, D \in K,$  
\begin{eqnarray*}
B \mbox{ lies between } A \mbox{ and }  D,  \\
C \mbox{ lies between } B \mbox{ and }  D,  
\end{eqnarray*}
then $C$ lies between  $A$ and  $D.$  
\end{remark}
\begin{proof}
The hypothesis and the triangle inequality yield the following chain:
\begin{eqnarray*}
d(A, D) = d(A,B) + d(B, D) = d(A, B) + d(B, C) + d(C, D) \\
\geq d(A, C) + d(C, D) \geq d(A, D).
\end{eqnarray*}
It follows $d(A,C) + d(C, B) = d(A, B),$ i.e. $C$ lies between $A$ and $D.$

\end{proof}
 \begin{lemma}
 \label{limit of geodesic's sequence}
 In  a compact geodesic space $(K, d),$ a sequence of the same-length geodesic segments with endpoints $A_n, B_n$ (for $n \in \mathbb{N}$) has a limit point that is a geodesic segment.
 \end{lemma}
 \begin{proof}
All these geodesic segments are the images of geodesic paths $\gamma_n \colon [0,l] \to K$. It means that each $\gamma_n$ is parametrized by the arc length hence all these geodesic paths $\gamma_n$ are equicontinuous. Since $K$ is compact, they are uniformly bounded. So, in accordance with Arzela--Askoli theorem, the sequence of $\gamma_n$ has a subsequence that converges to a continuous map $\gamma_*$. As all $\gamma_n$ are from the segment $[0, l]$ to~$K$ and also natural parametrized, then the map $\gamma_*$ is too from $[0, l]$ and is natural parametrized. 
Denote $A_* = \gamma_*(0), \ B_* = \gamma_*(l).$ Notice that the contrary assumption (i.e. $\gamma_*$ is not a geodesic path) implies the inequality $d(A_*, B_*) =l_* <l,$ that yields the existing of a number $n$ such that  $d(A_*, A_n) < (l-l_*)/2$ and $ d(B_*, B_n) < (l-l_*)/2.$ In the lights of this there exists a path from $A_n$ to $B_n$ through $A_*$ and $B_*$ the length of which is $l_* + d(A_*, A_n) +d(B_*, B_n) < l.$ It contradicts with the hypothesis and complete the proof.
 \end{proof}

\section{General Statement}
\subsection{Dynamics}
\label{sect of dyn}
Recall that players move in a metric space $(K, d).$ 
 We choose  restrictions of players' speeds   according to the original posedness and simultaneously to common sense.
Namely, we impose that players' trajectories $L \colon \mathbb{R}^+ \to K$ (Lion's one), $M \colon \mathbb{R}^+ \to K$ (Man's one) satisfy
$$ d \bigl(L(t_1), L(t_2)\bigr) \leq |t_1 - t_2|, \quad d \bigl(M(t_1), M(t_2)\bigr) \leq |t_1 - t_2|  \quad \forall t_{1}, t_{2} \geq 0.$$

 Lion wins if he can reduce the distance between him and Man down to any chosen in advance positive number $\varepsilon.$
 Man can choose his trajectory arbitrarily among those which satisfying the restrictions above. Moreover, we assume that he knows Lion's strategy and their initial positions so, as soon as we set $\varepsilon$ and players' initial positions, we will get both players' strategies and whole trajectories.  
Let us introduce Lion's strategy with a fixed $\varepsilon.$

\subsection{Description of the Lion's Strategy}
\label{sect of strategy}

Simple pursuit strategies are quite natural ones and seems to be interested for studying.  Our results are based on the discrete-time version of these strategies. More precise, we consider trajectories which are generated by $\varepsilon$-simple pursuit strategy as follows.
\begin{de}
A curve $\zeta(\cdot)= \bigl(L(\cdot), M(\cdot)\bigr)$ from $\mathbb{R}_{+}$ to $K^2,$ where $L(\cdot) $ and $M(\cdot) $ are 1-Lipschitz continuous,  is called the $\varepsilon$-simple-pursuit  curve  iff
for each $i \in \mathbb{N}\cup \{0\}$
the condition  $d\bigl(L(i\varepsilon), M(i\varepsilon)\bigr)\geq \varepsilon$ implies
\begin{eqnarray*}
d\bigl(L(i \varepsilon), L((i+1)\varepsilon)\bigr)=\varepsilon, \\
L\bigl((i+1)\varepsilon\bigr) \mbox{ lies between } L\bigl(i\varepsilon\bigr) \mbox{ and } M(i\varepsilon).
\end{eqnarray*}
\end{de}

So, the main idea of the Lion's strategy is in the taking aim to the current Man's position in  some time moments. 
Because of the (potential) non-uniqueness of geodesics in the space~$K,$ Man does not know which of geodesic segments Lion will choose on the each step. Simultaneously we want to suppose that Lion's strategy is predetermined  and Man can in advance calculate whole Lion's trajectory that would be generated by the Man's own moving. Therefore we propose a way which entails this determination. 

Let us denote by $\mathcal{G}_{AB}$ the set of all geodesic segments with some fixed endpoints $A, B \in K. $ For all pair of points from $K$ such a set is non-empty  (for the same points we get that the geodesic segment degenerates into a point). Then let us define a set $\mathcal{G}$ using the axiom of choice as follows: $\mathcal{G}$ is a subset of the set of all geodesic segments from $K$ and $\mathcal{G}$ has only one element in common with each of $\mathcal{G}_{AB}$ for all $A, B \in K.$ The set $\mathcal{G}$ meaningfully is Lion's possible paths.

Let us explain  the announced  Lion's strategy in detail. While the distance between players is more than $\varepsilon$ let Lion move in the following way.
Firstly, at the initial time moment, Lion knows players' initial positions $L(0), M(0)$ and till the moment $\varepsilon$ he should move along the geodesic segment with endpoints $L(0)$ and $M(0)$ from the set $\mathcal{G}.$ Secondly, at time $\varepsilon,$ he should find the geodesic segment with endpoints $L(\varepsilon)$ and $M(\varepsilon)$ from the set $\mathcal{G}$ and move along it till $2\varepsilon$ time moment. Etcetera, at time $i\varepsilon$ (where $i \in \mathbb{N}$) Lion changes his moving according to the current players' positions.
Cause we will consider players' trajectories as curves from $\mathbb{R}$ to $K$ then we  should  consider the formally case $d\bigl(L(i\varepsilon), M(i\varepsilon)\bigr) < \varepsilon.$ In this case, Lion analogise should  move along the geodesic segment with endpoints $L(i\varepsilon)$ and $M(i\varepsilon)$ from the set $\mathcal{G},$ but when he manages the point $M(i\varepsilon)$ he should stay in this point till $(i+1)\varepsilon$ time moment. 

Thus, this strategy   can be realised by force of a choice of $\mathcal{G}.$ Obviously, this Lion's strategy  generates $\varepsilon$-simple-pursuit curve independent of Man's strategy.

\subsection{Main results}
\label{sect of results}

Studying $\varepsilon$-simple pursuit curves, we obtained the following result:
\begin{proposition}
\label{th2}
If  a compact geodesic space $(K,d)$  satisfies the betweenness property, then, for each $\varepsilon$-simple-pursuit curve $\zeta(\cdot) = \bigl(L(\cdot), M(\cdot)\bigr) \colon \mathbb{R}_+ \to K^2,$  there exists a number $T>0$ such that $d\bigl(L(T), M(T)\bigr)< \varepsilon.$
\end{proposition}
Proof of Proposition~\ref{th2} takes the whole Section~\ref{sect-proof}. However, since we consider $\varepsilon$-capture and have Lion's strategy generating a $\varepsilon$-simple-pursuit curve (see the previous subsection), it means this proposition directly implies the following result for lion and man game:
\begin{theorem}
\label{th betw}
In a compact geodesic space satisfying the betweenness property Lion has a winning non-anticipative discrete-time strategy.
\end{theorem}

There are many special spaces satisfying betweenness property: Busemann convex spaces (by Proposition 8.2.4 from \cite{papadopoulos}), Ptolemy spaces (by Proposition 3.3 from \cite{Nicolae}), geodesic spaces with convex metric (by Proposition 3.4 from \cite{Nicolae}). Let us point out one of them.
\begin{Corollary}
\label{th convex}
  Lion also has a winning non-anticipative discrete-time strategy in a compact geodesic space  with a convex metric.
\end{Corollary}

\section{Proof of Theorem 1}
\label{sect-proof}
Suppose the contrary.
Namely, we suppose that there exists a positive number $\varepsilon$  and Man's trajectory~$\mathcal{M}$ such that chosen strategy generates Lion's trajectory~$\mathcal{L}$ and consequently the $\varepsilon$-simple-pursuit curve  $\zeta(\cdot)= \bigl(\mathcal{L}(\cdot), \mathcal{M}(\cdot)\bigr)$ with the property $d\bigl(\mathcal{L}(t), \mathcal{M}(t)\bigr)>\varepsilon$ for all $ t\geq 0.$ These special denotation will be used in subsections \ref{sect of rounds}--\ref{sect of contradiction}, whereas in subsection \ref{sect of good} and \ref{sect of behaviour} we will prove supporting statements unrelated to the curve $\zeta$.  
In subsection \ref{sect of contradiction} we will show that this assumption implies a contradiction.

Since we fixed the number $\varepsilon, $ let us denote the sequence of control correction moments by $\Delta = \{\tau_i\}_{i=0}^{\infty} = \{i\varepsilon\}_{i=0}^{\infty}.$ Moreover it is convenient to set the metric $\rho$ on $K^2$ as follows:
  for all $ A = (A_1, A_2)\in K^2, \  B=(B_1, B_2) \in K^2$ (and respectively $A_1, A_2, B_1, B_2 \in K$) 
  $$\rho(A,B) = max \bigl\{d(A_{1},\ B_{1}), d(A_{2},\ B_{2})\bigr\}.$$

\subsection{Good Curves}
\label{sect of good}

By our definition, $\zeta(\cdot)=\bigl(\mathcal{L}(\cdot), \mathcal{M}(\cdot)\bigr)$ is the $\varepsilon$-simple-pursuit curve along which the players will go.
But we want to consider   segments of any $\varepsilon$-simple-pursuit curves with `no $\varepsilon$-capture' property.
 Let us formulate it rigorously.
\begin{de} A curve $\gamma(\cdot) = \bigl(L(\cdot), M(\cdot)\bigr) : [\tau_a, \tau_b] \rightarrow K^2 $ where  $\tau_a, \tau_b \in \Delta \cup \{\infty\} , \ \tau_a <\tau_b$ is called {\it a good curve} iff
\begin{flalign*}
1.\ & L(\cdot), M(\cdot) \mbox{ are 1-Lipschitz continuous, } &\\
2. \ & L(\tau_{i+1}) \mbox{ lies between } L(\tau_i) \mbox{ and } M(\tau_i) &&\mbox{for \ }\tau_a \leq \tau_i < \tau_b,&\\
3.\ & d\bigl(L(\tau_i), L(\tau_{i+1})\bigr) = \varepsilon   &&\mbox{for \ }\tau_a \leq \tau_i < \tau_b,& \\
4.\ & d\bigl(L(\tau_i),M(\tau_{i})\bigr) \geq \varepsilon   &&\mbox{for \ }\tau_a \leq \tau_i \leq \tau_b.&
\end{flalign*}
\end{de}
So, we will consider good curves properties and employ them to describing  behaviour of the considering curve $\zeta(\cdot).$ 

\begin{lemma}
	\label{limit of good curves}
	A sequence of good curves
	$$\psi^n(\cdot) = \bigl(\psi^n_L(\cdot), \psi^n_M(\cdot) \bigr): [0, \tau_i] \rightarrow K^2$$
	has a subsequence that converges to a good curve $$\psi(\cdot) = \bigl(\psi_L(\cdot), \psi_M(\cdot) \bigr): [0, \tau_i] \rightarrow K^2.$$
\end{lemma}
\begin{proof}
	Since each $\psi^n_L(\cdot)$ and $\psi^n_M(\cdot)$ are 1-Lipschitz continuous, each $\psi^n(\cdot)$ is 1-Lipschitz continuous too (by the definition of the metric $\rho$). In addition, $K^2$ is compact hence all of $\psi^n(\cdot)$ are both equicontinuous and uniformly bounded. So this sequence has a limit point by Arzela--Askoli theorem. Note that $\psi_L(\cdot)$ and $\psi_M(\cdot)$ are 1-Lipschitz continuous, as well as each of $\psi^n_L(\cdot)$ and $\psi^n_M(\cdot)$. In the same way we get Item 4 of good curves' definition. The remaining items follow from~Lemma~\ref{limit of geodesic's sequence}. Thus $\psi(\cdot)$ is a good curve indeed.
	
\end{proof}

\subsection{Behaviour of Distance between Players}
\label{sect of behaviour}

 In this subsection we illustrate properties of good curves.
  Let us denote by $d_{\gamma}(t)$ the distance between the components of a good curve $\gamma(\cdot)$ at time $t.$

Proposition~\ref{lem-1} uses the idea of the triangle inequality from \cite{alexander2010total}.
\begin{proposition}
\label{lem-1}
Let $(K,d)$ be a geodesic space satisfying the betweenness property. Let $\gamma(\cdot)=\bigl(L(\cdot), M(\cdot)\bigr) \colon [\tau_a, \tau_b] \to K$ be a good curve, $[\tau_i, \tau_{i+1}] \subset [\tau_a,\tau_b].$
The following statements are equivalent:
\begin{enumerate}
\item{$\forall t \in [\tau_i, \tau_{i+1}] \ \dgam(t) = \dgam(\tau_i),$}
\item{$\dgam(\tau_i) = \dgam(\tau_{i+1}),$}
\item{$d\bigl(M(\tau_i), M(\tau_{i+1})\bigr)=\varepsilon$ and $ M(\tau_i)$ lies between  $L(\tau_i)$ and $ M(\tau_{i+1})$ }
\end{enumerate}
\end{proposition}
\begin{proof}
Statement 2 trivially follows from statement 1.

Then let us show that statement 1 follows from statement 3. The statement~3 implies that $M(t)$ for $t $ from $\tau_i$ to  $\tau_{i+1}$ is a geodesic path joining $M(\tau_i)$ to $M(\tau_{i+1}).$ Consider an arbitrary $t \in [\tau_i, \tau_{i+1}].$ Then we get that $M(t)$ lies between $M(\tau_i)$ and $M(\tau_{i+1}).$  By this, hypothesis of statement~3 and Remark~\ref{geod-trans} we obtain that $ M(t)$ lies between  $L(\tau_i)$ and $ M(\tau_{i+1})$ too. From the definition of good curves we also know that $ L(t)$ lies between  $L(\tau_i)$ and $ M(\tau_{i+1})$ for all  $t \in [\tau_i, \tau_{i+1}].$ Thus, $L(t)$ lies between $L(\tau_i)$ and $M(t)$ too, hence by Remark~\ref{geod-trans} 
\begin{eqnarray*}
d_{\gamma}(t) = d\bigl(L(t), M(t)\bigr) =  d\bigl(L(\tau_i), M(t)\bigr) - d\bigl(L(\tau_i), L(t)\bigr)  \\
 = d\bigl(L(\tau_i), M(t)\bigr) - (t-\tau_i) =  d\bigl(L(\tau_i), M(\tau_i)\bigr) + d\bigl(M(\tau_i), M(t)\bigr) - (t-\tau_i) \\
  = d\bigl(L(\tau_i), M(\tau_i)\bigr) + (t-\tau_i) - (t-\tau_i) \\
  = d\bigl(L(\tau_i), M(\tau_i)\bigr)    = d_{\gamma}(\tau_i),
\end{eqnarray*} 
i.e. statement 1 follows from statement 3.

Let us assume statement 2 and show that statement 3 holds.
Note that, by the triangle inequality and the definition of good curves, we obtain
 $$\dgam(\tau_{i+1}) =
 d\bigl(L(\tau_{i+1}), M(\tau_{i+1})\bigr)$$
 $$ \leq
 d\bigl(L(\tau_{i+1}), M(\tau_{i})\bigr) + d\bigl(M(\tau_{i}), M(\tau_{i+1})\bigr) \leq
 d\bigl(L(\tau_{i+1}), M(\tau_{i})\bigr) + \varepsilon $$
 $$=  d\bigl(L(\tau_{i+1}), M(\tau_{i})\bigr) + d\bigl(L(\tau_i), L(\tau_{i+1})\bigr) =
  d\bigl(L(\tau_i), M(\tau_{i+1})\bigr) = \dgam(\tau_i).$$
  Due to the equality from statement 2, both these non-equality signs should be changed to equality ones, so we get $d(M(\tau_i), M(\tau_{i+1})) = \varepsilon$ (as wanted) and the equation $d\bigl(L(\tau_{i+1}), M(\tau_{i+1})\bigr) =
 d\bigl(L(\tau_{i+1}), M(\tau_{i})\bigr) + d\bigl(M(\tau_{i}), M(\tau_{i+1})\bigr) $ that implies only $M(\tau_i)$ lies between $L(\tau_{i+1})$ and $M(\tau_{i+1}).$ But by definition of good curves we also have that $L(\tau_{i+1})$ lies between  $L(\tau_i)$ and $M(\tau_i),$ hence Remark~\ref{geod-trans} gives the wanted relation: $M(\tau_i)$ lies between $L(\tau_i)$ and $M(\tau_{i+1}).$

\end{proof}

As a corollary we get the following.
\begin{remark}
\label{decrease}
\it For every good curve $\gamma$ the distance $\dgam(\cdot)$ does not increase.
\end{remark}

\begin{proposition}
\label{lem-2.2}
Let $(K,d)$ be a geodesic space satisfying the betweenness property.
If $\gamma(\cdot)=\bigl(L(\cdot), M(\cdot)\bigr) \colon [0, \tau_n] \to K$ be a good curve such that 
$\dgam(0) = \dgam(\tau_n)$ then $L(\cdot)$ on $[0, \tau_n]$ is a geodesic path.
\end{proposition}

\begin{proof}

Cause Lion has to move along geodesic segments all his steps it is enough to show that the points $L(\tau_0), L(\tau_1),\ldots,L(\tau_n)$ are in the following relation: $L(\tau_{n-i})$ lies between $L(\tau_{n-i-1})$ and $L(\tau_n)$ for all $i, \ 1\leq i \leq n-1.$ These relations allow us to calculate $d\bigl(L(0), L(\tau_n)\bigr)$ as the sum of $d\bigl(L(\tau_{i-1}), L(\tau_i)\bigr),$ then, since each of the addends equals $\varepsilon$ we get   $d\bigl(L(0), L(\tau_n)\bigr) = n\varepsilon$ and consequently  $L(\cdot)$ on $[0, \tau_n]$ is a geodesic path.

Further, in this proof we will use denotations $[A, B] \subset [A, C]$ or $[B, C] \subset [A, C]$ instead of the phrase `$B$ lies between $A$ and $C$' cause it  simplifies the narration. Thus, we need to proof that   
\begin{eqnarray*}
[L(\tau_{n-1}),L(\tau_{n})] \subset [L(\tau_{n-2}),L(\tau_{n})] \subset \ldots \subset [L(\tau_1), L(\tau_{n})] \subset [L(0),L(\tau_{n})].
\end{eqnarray*}
But we will prove a more strong statements: 
\begin{eqnarray*}
\label{2.2.0}
\ d\bigl(M(\tau_i), M(\tau_{i+1})\bigr)=\varepsilon, \quad  0\leq i \leq n-1, \\
\ [L(\tau_{n}),M(\tau_{n})] \subset [L(\tau_{n-1}),M(\tau_{n})] \subset \ldots \subset [L(\tau_1), M(\tau_n)] \subset [L(0),M(\tau_{n})],\\
\ [L(0),M(0)] \subset [L(0),M(\tau_{1})] \subset \ldots  \subset [L(0), M(\tau_{n-1})] \subset [L(0),M(\tau_{n})].
\end{eqnarray*}
These relations involves  that the images of $[0, \tau_n]$ under both $L(\cdot)$ and $M(\cdot)$ are geodesic segments.

If $n=1,$ the proof is trivial by virtue of the chosen Lion's strategy.  It is a basis of induction.

Let us assume that these inclusions hold for some natural $k$. Prove the step of the induction --- the case $n=k+1.$

To do it let us define a curve $\eta(\cdot) = \bigl(\eta_L(\cdot), \eta_M(\cdot)\bigr) : [0, k\varepsilon] \rightarrow K^2$ as follows:
$$\eta_L(t)=L(t), \quad \eta_M(t)=M(t+\tau_{1}) \quad \forall t \in [0, k\varepsilon].$$
It is easy to see that $\eta(\cdot)$  satisfies Items 1 and 2 from the definition of good curves.
Consider Item 3. The definition of good curves and Proposition~\ref{lem-1}  give us the following for all $m, \ 0\leq m \leq k-1,$
\begin{eqnarray*}
[\eta_L(m\varepsilon), \eta_L((m+1)\varepsilon)]=[L(m\varepsilon), L((m+1)\varepsilon)] \ \\
\subset[L(m\varepsilon), M(m\varepsilon)] \subset [L(m\varepsilon), M((m+1)\varepsilon)] \ \\
=[\eta_L(m\varepsilon), \eta_M(m\varepsilon)].
\end{eqnarray*}
Hence, Item 3 of this definition holds too.
Let us check the last item and, at the same time, show the equality $d_{\eta}(0)=d_{\eta}(k\varepsilon).$
By the inductive hypothesis and Remark \ref{decrease} we obtain 
\begin{eqnarray*}
d_{\eta}(0)
&=&d\bigl(L(0),M(\tau_{1})\bigr) \\
&=&d\bigl(L(\tau_{1}),M(\tau_{1})\bigr)+\varepsilon \\
&=&\dgam(\tau_1)+\varepsilon \\
&=& \dgam(\tau_i)+\varepsilon
=d_{\eta}(i\varepsilon);
\end{eqnarray*}
moreover, the restriction $\gamma|_{[\tau_{1}, \tau_{k+1}]}$ satisfies the induction hypothesis for the case $n=k$, hence $L|_{[\tau_{1}, \tau_{k+1}]}$ and $M|_{[\tau_{1}, \tau_{k+1}]}$ are geodesic paths too; that implies
\begin{eqnarray*}
d_{\eta}((k-1)\varepsilon) = &d&\bigl(L(\tau_{k-1}),M(\tau_{k})\bigr)=d\bigl(L(\tau_{k-1}),M(\tau_{k-1})\bigr)+\varepsilon \\
= &d&\bigl(L(\tau_{k}),M(\tau_{k})\bigr)+\varepsilon = d\bigl(L(\tau_{k}),M(\tau_{k+1})\bigr)=d_{\eta}\bigl(k\varepsilon\bigr).
\end{eqnarray*}
Thus, $\eta(\cdot)$ is a good curve.

So, the condition $d_{\eta}(0)=d_{\eta}(k\varepsilon)$ allows us to use induction hypothesis for the case $n=k$ to curve $\eta(\cdot).$ In this way, we get
\begin{eqnarray*}
[\eta_L(k\varepsilon), \eta_M(k\varepsilon)] \subset [\eta_L((k-1)\varepsilon), \eta_M(k\varepsilon)] \subset \ldots \subset [\eta_L(0), \eta_M(k\varepsilon)], \\
\ [\eta_L(0), \eta_M(0)] \subset [\eta_L(0), \eta_M(\varepsilon)] \subset \ldots \subset [\eta_L(0), \eta_M(k\varepsilon)].
\end{eqnarray*}

Substituting $\gamma(\cdot)$ for $\eta(\cdot),$ we obtain
\begin{eqnarray*}
[L(\tau_{k}), M(\tau_{k+1})] \subset [L(\tau_{k-1}), M(\tau_{k+1})] \subset \ldots \subset [L(0), M(\tau_{k+1})], \\
\ [L(0), M(\tau_1)] \subset [L(0), M(\tau_{2})] \subset \ldots \subset [L(0), M(\tau_{k+1})].
\end{eqnarray*}
The remaining inclusions $[L(\tau_{k+1}), M(\tau_{k+1})] \subset [L(\tau_k), M(\tau_{k+1})]$ and $[L(0), M(0)] \subset [L(0), M(\tau_{1})]$ follow from the definition of good curves and Proposition~\ref{lem-1} respectively.

The equality $d\bigl(M(\tau_i), M(\tau_{i+1})\bigr)=\varepsilon$ for $i, \ 0 \leq i \leq n-1,$ is given by $\gamma|_{[\tau_{0}, \tau_{k}]}$ and $\gamma|_{[\tau_{1}, \tau_{k+1}]}$ that satisfy induction hypothesis for the case $n=k$.
So, we get what we need.

Thus, we proved this proposition for all natural $n.$

\end{proof}

\subsection{Rounds}
\label{sect of rounds}

Recall that $\zeta$ is the $\varepsilon$-simple-pursuit curve that we chose in the beginning of the proof. 
We shall add the following useful construction.
\begin{de} The restriction of $\zeta(\cdot)$ to an interval $[\tau_{i}, \tau_{j}]$  is called a round for a set $A \subset K^2$ iff
\begin{enumerate}
\item {$\tau_i, \tau_j \in \Delta;$}
\item {$\tau_j-\tau_i>\varepsilon$, i.e. $j-i>1;$}
\item{ $\zeta ( \tau _{i} ) \in A;$}
\item{ $\zeta ( \tau _{j} ) \in A;$}
\item{ $\zeta ( \tau _k ) \notin A$ for any natural $k, \ i<k<j.$}
\end{enumerate}
\end{de}

\begin{lemma}
\label{lemma-r1}
 Let $A$  be a closed $\frac{\varepsilon}{3}$-neighbourhood of a limit point $Z^*$ of the sequence $\{\zeta(\tau_n)\}_{n=1}^{\infty};$ then, there are countably many rounds for $A.$
\end{lemma}
\begin{proof}
Indeed, there are countably many points from the set $\{\zeta(\tau_{n})\ | \ n \in  \mathbb{N} \} \cap A$. But there is no  natural $i$ such that both $\zeta(\tau_i)$ and $\zeta(\tau_{i+1})$ belong to $A$ because
\begin{eqnarray}
\label{r1}
\rho\bigl(\zeta(\tau_i), \zeta(\tau_{i+1})\bigr) \geq d\bigl(\mathcal{L}(\tau_i), \mathcal{L}(\tau_{i+1})\bigr)=\varepsilon> diam(A).
\end{eqnarray}
 Then, each point like $\zeta(\tau_k) \in A$ is the start of the corresponding round and the end of previous round at the same time. So, we have countably many rounds for this $A$.
\end{proof}
\begin{proposition}
\label{sequence of rounds}
There exists a nonempty set $S \subset K^2 \ \bigl(diam(S)<\varepsilon\bigr) $ and a natural number $m \geq2$ such that there is a sequence of rounds for the set $S$ such that lengths of each of rounds are $m\varepsilon.$

\end{proposition}

\begin{proof}

Let  $F$ be a limit point of the sequence $\{\zeta(\tau_n)\}_{n=1}^{\infty}.$ Consider the set of all rounds for the closed ball $B_{\frac{\varepsilon}{3}}(F),$ i.e. the closed $\frac{\varepsilon}{3}$-neighbourhood of the  point $F$.

If we find countably many rounds such that their lengths are uniformly bounded from above, then we can find  desired same-legth rounds. Suppose the contrary.  

In this case, we have a sequence of rounds for $B_{\frac{\varepsilon}{3}}(F)$ the lengths of which grow infinitely. We shall show that there exists another required set.

Here we introduce the term {\it cage}.
Consider a finite covering of $K^2$ composed of $B_{\frac{\varepsilon}{3}}(F)$ and other $\frac{\varepsilon}{3}$-balls. Let $B_{\frac{\varepsilon}{3}}(F)$ be the ball number 1. Let us enumerate other balls from the covering as $2, 3,\ldots,N$.
{\it Cages} are defined as follows. Each point in $K^2$ gets a number that is the minimal number among the numbers of the balls from the covering containing this point. Let the $i^{th}$~cage be the set of all points that marked with the number~$i$;  this set may be empty. Note that we obtain $N$ cages and the first one coincides with the set $B_{\frac{\varepsilon}{3}}(F)$.

Let us regard each round $\zeta|_{[\tau_a, \tau_b]}$ for $B_{\frac{\varepsilon}{3}}(F)$ as a tuple $s_{a},\ldots,s_{b},$ where $s_{i} \in \{1, 2,\ldots,N\}$ is the number of the cage that contains the point $\zeta(\tau_i)$. Note that $s_{a}=s_{b}=1 $ and $ s_{i} \neq 1$ for all $i,\ a<i<b.$ 
Note that  inequalities~(\ref{r1}) hold, consequently
 two neighbours $\tau_i$ and $\tau_{i+1}$ belong to different cages; hence, in such a tuple all neighbour symbols should be different.

Since we have at most $N$ cages, it follows that in each tuple there exist two equal numbers among the first $N+1$ symbols;   if the length of a tuple is less than $N+1,$ the first and the last elements of this tuple, for instance, are such equal numbers.   Due to the paragraph above, these equal numbers can not be neighbours in a tuple. Thus, we can consider `subtuples' between the nearest equal numbers instead of the whole tuples. We get that the lengths of these subtuples are more than 2 and less than $N+2$. Since we have countably many subtuples with uniformly bounded lengths, we can find countably many subtuples with the same-legth. In the same way each subtuple starts with a symbol from $\{1, 2,\ldots,N\}$. So we can find  countably many subtuples starting with the same number, say number $k$. It means that the corresponding curves start and finish in the $k^{th}$ cage. Thus, we have countably many same-length restrictions of $\zeta(\cdot)$. Note that they are rounds for the $k^{th}$ cage. Recall that the diameter of $k^{th}$ cage is not more than~$\varepsilon$.

Thus, we obtain countably many same-length rounds for the required set.
\end{proof}

\subsection{Limit Curve}
\label{sect of limit}

Let $\{\zeta|_{[\tau_{i_n}, \tau_{i_n}+m\varepsilon]} \}_{n=1}^{\infty}$ be a sequence of rounds for a set $S$ from Proposition~\ref{sequence of rounds}.
Consider the sequence of $\zeta^n(\cdot) = \bigl(\zeta^n_L(\cdot), \zeta^n_M(\cdot)\bigr): [0, m\varepsilon] \rightarrow K^2$ such that $$\zeta^n(t)=\zeta(\tau_{i_n}+t) \qquad \forall t \in [0, m\varepsilon].$$

Since $\zeta^n_L(\cdot)$ and $\zeta^n_M(\cdot)$ are 1-Lipschitz curves and thanks to Arzela--Askoli theorem, this sequence of $\zeta^n$ has a subsequence converging to a continuous map. Let the corresponding subsequences of $\zeta^n_L(\cdot)$ and $\zeta^n_M(\cdot)$ converge to continuous maps $\zeta_L^{\ast}(\cdot)$ and $\zeta_M^{\ast}(\cdot)$. Set $\zeta^{\ast}(\cdot) = \bigl(\zeta_L^{\ast}(\cdot), \zeta_M^{\ast}(\cdot)\bigr).$

\begin{proposition}
\label{limit}
The following statements hold:  \\
1. $\zeta^{\ast}(\cdot)$ is a good curve; \\
2. $d\bigl(\zeta_L^{\ast}(0), \zeta_L^{\ast}(m\varepsilon)\bigr) <\varepsilon$; \\
3. $d\bigl(\zeta_L^{\ast}(0), \zeta_M^{\ast}(0)\bigr)=d\bigl(\zeta_L^{\ast}(m\varepsilon), \zeta_M^{\ast}(m\varepsilon)\bigr).$
\end{proposition}
\begin{proof}
1. The first statement follows from Lemma~\ref{limit of good curves}.

2. Since all $\zeta_n(0)$ and $\zeta_n(m\varepsilon)$ belong to $S,$ it follows that $\zeta^{\ast}(0)$ and $\zeta^{\ast}(m\varepsilon)$ belong to the closure of this set. Hence, $$d\bigl(\zeta^{\ast}_L(0), \zeta^{\ast}_L(m\varepsilon)\bigr) \leq \rho \bigl(\zeta^{\ast}(0), \zeta^{\ast}(m\varepsilon)\bigr) \leq diam(S) < \varepsilon.$$

3. We know that $d\bigl(\mathcal{L}(\tau_k), \mathcal{M}(\tau_k)\bigr) - d\bigl(\mathcal{L}(\tau_{k+m}), \mathcal{M}(\tau_{k+m})\bigr)\rightarrow 0$ as $k \rightarrow \infty$ because the distance between the players  does not increase and is positive.
Then, $d\bigl(\zeta_L^{n}(0), \zeta_M^{n}(0)\bigr)-d\bigl(\zeta_L^{n}(m\varepsilon), \zeta_M^{n}(m\varepsilon)\bigr) \rightarrow 0$ as $n \rightarrow \infty.$ Thus, $d\bigl(\zeta_L^{\ast}(0), \zeta_M^{\ast}(0)\bigr)=d\bigl(\zeta_L^{\ast}(m\varepsilon),\zeta_M^{\ast}(m\varepsilon)\bigr).$

\end{proof}

\subsection{The Last Component}
\label{sect of contradiction}

By this subsection we only need to show the contradiction to the original assumption.
Recall that we assumed that there exists an $\varepsilon$-simple-pursuit curve $\zeta \colon \mathbb{R}_{+} \to K^2$ which free of $\varepsilon$-capture.

From Proposition~\ref{lem-2.2} and the last statement of Proposition~\ref{limit},
the image of $[0, m\varepsilon]$ under $\zeta_L^{\ast}(\cdot)$ is a geodesic segment with endpoints $\zeta_L^{\ast}(0)$ and  $\zeta_L^{\ast}(m\varepsilon)$.
Hence, we can calculate the distance $d\bigl(\zeta_L^{\ast}(0), \zeta_L^{\ast}(m\varepsilon)\bigr)$ with help of the points lying between $\zeta_L^{\ast}(0)$ and $\zeta_L^{\ast}(m\varepsilon):$
\begin{eqnarray*}
d\bigl(\zeta_L^{\ast}(0), \zeta_L^{\ast}(m\varepsilon)\bigr)
&=&d\bigl(\zeta_L^{\ast}(0), \zeta_L^{\ast}(\varepsilon)\bigr)+d\bigl(\zeta_L^{\ast}(\varepsilon), \zeta_L^{\ast}(m\varepsilon)\bigr) \\
&=&d\bigl(\zeta_L^{\ast}(0), \zeta_L^{\ast}(\varepsilon)\bigr)+d\bigl(\zeta_L^{\ast}(\varepsilon), \zeta_L^{\ast}(2\varepsilon)\bigr)+d\bigl(\zeta_L^{\ast}(2\varepsilon), \zeta_L^{\ast}(m\varepsilon)\bigr) \\
& \ldots & \\
&=&d\bigl(\zeta_L^{\ast}(0), \zeta_L^{\ast}(\varepsilon)\bigr)+\ldots+d\bigl(\zeta_L^{\ast}((m-1)\varepsilon), \zeta_L^{\ast}(m\varepsilon)\bigr) \\
&=&m\varepsilon
\geq 2\varepsilon,
\end{eqnarray*}
where the inequality follows from Proposition~\ref{sequence of rounds}. 

But it contradicts with  Statement 2 of Proposition~\ref{limit}. Thus, the assumption is false. In other words, at some time the distance between players will be less than $\varepsilon$ and $\varepsilon$-capture will take place.



\begin{thebibliography}{20}

\bibitem[{Alexander et~al(2010)}]{alexander2010total}
Alexander S, Bishop R, Ghrist R (2010) Total curvature and simple pursuit on
  domains of curvature bounded above. Geometriae Dedicata 149(1):275--290

\bibitem[{Alonso et~al(1992)}]{alonso1992lion}
Alonso L, Goldstein AS, Reingold EM (1992) Lion and man: Upper and lower
  bounds. ORSA Journal on Computing 4(4):447--452

\bibitem[{Ba{\v{c}}{\'a}k(2012)}]{bavcak2012note}
Ba{\v{c}}{\'a}k M (2012) Note on a compactness characterization via a pursuit
  game. Geometriae Dedicata 160(1):195--197

\bibitem[{Barmak (2017)}]{Barmak.2017}
Barmak, J. A. (2017). Lion and man in non-metric spaces.  arXiv preprint arXiv:1703.01480

\bibitem[{Beveridge and Cai(2015)}]{beveridge2015two}
Beveridge A, Cai Y (2015) Two-dimensional pursuit-evasion in a compact domain
  with piecewise analytic boundary. arXiv preprint arXiv:150500297

\bibitem[Bollob\'{a}s at al (2012)]{Bollobas.2012}
Bollob\'{a}s, B., Leader, I., \& Walters, M. (2012). Lion and man --- can both win?  Israel Journal of Mathematics, 189(1), pp 267--286

\bibitem[{Bramson et~al(2014)}]{bramson2014rubber}
Bramson M, Burdzy K, Kendall WS (2014) Rubber bands, pursuit games and shy
  couplings. Proceedings of the London Mathematical Society

\bibitem[{Bridson and Haefliger(2011)}]{bridson2011metric}
Bridson MR, Haefliger A (2011) Metric spaces of non-positive curvature, vol
  319. Springer Science \& Business Media

\bibitem[{Chernous'~ko(1976)}]{chernous1976problem}
Chernous'~ko F (1976) A problem of evasion from many pursuers. Journal of
  Applied Mathematics and Mechanics 40(1):11--20

\bibitem[{Isaacs(1965)}]{isaacs1999differential}
Isaacs, R. (1965). Differential games, a mathematical theory with applications to optimization, control and warfare.


\bibitem[{Isler and Karnad(2009)}]{karnad2009lion}
Isler V, Karnad N (2009) Lion and man game in the presence of a circular
  obstacle. In: 2009 IEEE/RSJ International Conference on Intelligent Robots
  and Systems, IEEE, pp 5045--5050

\bibitem[{Isler and Noori(2015)}]{noori2015lion}
Isler V, Noori N (2015) The lion and man game on convex terrains. In:
  Algorithmic Foundations of Robotics XI, Springer, pp 443--460

\bibitem[{Ivanov and Ledyaev(1983)}]{ivanov1981optimality}
Ivanov R, Ledyaev YS (1983) Optimality of the pursuit time in a differential
  game with several pursuers under simple motion. Proceedings of the Steklov
  Institute of Mathematics 158, pp 93--103

\bibitem[{Kumkov et al (2017)}]{kumkov}
Kumkov SS, Le M\'{e}nec S \& Patsko VS (2017). Zero-Sum Pursuit-Evasion Differential Games with Many Objects: Survey of Publications. Dynamic Games and Applications, 7(4), pp 609--633

\bibitem[{Littlewood(1953)}]{littlewood1986mathematician}
Littlewood, JE (1953). A mathematicians miscellany, Methuen \& Co. Ltd., London.

\bibitem[{Nicolae(2013)}]{Nicolae} 
Nicolae, A (2013). Asymptotic behavior of averaged and firmly nonexpansive mappings in geodesic spaces. Nonlinear Analysis: Theory, Methods \& Applications, 87, pp 102--115.

\bibitem[{O'Kane and Stiffler(2012)}]{stiffler2012shortest}
O'Kane JM, Stiffler NM (2012) Shortest paths for visibility-based
  pursuit-evasion. In: Robotics and Automation (ICRA), 2012 IEEE International
  Conference on, IEEE, pp 3997--4002

\bibitem[{Papadopoulos (2005)}]{papadopoulos}
Papadopoulos A (2005) Metric spaces, convexity and nonpositive curvature. European Mathematical Society 6

\bibitem[{Petrosjan(1993)}]{petrosjan1993differential}
Petrosjan LA (1993) Differential games of pursuit, vol~2. World Scientific

\bibitem[{Pontryagin(1966)}]{pontryagin1966theory}
Pontryagin LS (1966) On the theory of differential games. Russian Mathematical
  Surveys 21(4):193--246

\bibitem[{Sgall(2001)}]{sgall2001solution}
Sgall J (2001) Solution of david gale's lion and man problem. Theoretical
  Computer Science 259(1):663--670

\bibitem[{Tovar and LaValle(2008)}]{tovar2008visibility}
Tovar B, LaValle SM (2008) Visibility-based pursuit-evasion with bounded
  speed. The International Journal of Robotics Research 27(11-12):1350--1360

\end{thebibliography}

\end{document}